\documentstyle[amscd,amssymb,xypic,verbatim,11pt]{amsart}
\xyoption{all}

\newcommand{\nc}{\newcommand}
\newcommand{\rnc}{\renewcommand}

\nc{\exto}[1]{\stackrel{#1}{\longrightarrow}}
\nc{\dlim}{{\mathop{\lim\limits_{\longrightarrow}}}}
\nc{\lan}{\big\langle}
\nc{\ran}{\big\rangle}

\nc{\kk}{{\mathsf{k}}}
\nc{\ix}{{\mathsf{i}}}
\nc{\jx}{{\mathsf{j}}}

\nc{\C}{{\mathbb{C}}}
\nc{\HH}{{\mathbb{H}}}
\nc{\LL}{{\mathbb{L}}}
\nc{\PP}{{\mathbb{P}}}
\nc{\RR}{{\mathbb{R}}}
\nc{\QQ}{{\mathbb{Q}}}
\nc{\ZZ}{{\mathbb{Z}}}

\nc{\CA}{{\mathcal{A}}}
\nc{\CB}{{\mathcal{B}}}
\nc{\CC}{{\mathcal{C}}}
\nc{\D}{{\mathcal{D}}}
\nc{\CE}{{\mathcal{E}}}
\nc{\CF}{{\mathcal{F}}}
\nc{\CG}{{\mathcal{G}}}
\nc{\CH}{{\mathcal{H}}}
\nc{\CJ}{{\mathcal{J}}}
\nc{\CL}{{\mathcal{L}}}
\nc{\CM}{{\mathcal{M}}}
\nc{\CN}{{\mathcal{N}}}
\nc{\CO}{{\mathcal{O}}}
\nc{\CQ}{{\mathcal{Q}}}
\nc{\CR}{{\mathcal{R}}}
\nc{\CS}{{\mathcal{S}}}
\nc{\CT}{{\mathcal{T}}}
\nc{\CU}{{\mathcal{U}}}
\nc{\CV}{{\mathcal{V}}}
\nc{\CW}{{\mathcal{W}}}
\nc{\CX}{{\mathcal{X}}}
\nc{\CY}{{\mathcal{Y}}}
\nc{\CZ}{{\mathcal{Z}}}
\nc{\CMo}{{\mathcal{M}^\circ}}
\nc{\Co}{{{C}^\circ}}

\nc{\BY}{{\overline{Y}}}
\nc{\BZ}{{\overline{Z}}}
\nc{\BYD}{{\overline{Y}{}^{|D|}}}
\nc{\OZ}{{\overline{Z}}}
\nc{\bg}{{\bar{g}}}

\nc{\bq}{{\mathbf{q}}}
\nc{\BD}{{\mathbf{D}}}
\nc{\BG}{{\mathbf{G}}}
\nc{\BL}{{\mathbf{L}}}
\nc{\BM}{{\mathbf{M}}}
\nc{\BP}{{\mathbf{P}}}
\nc{\BPr}{{\mathsf{P}}}
\nc{\BR}{{\mathbf{R}}}
\nc{\BRO}[1]{{{\mathbf{R}}^{\circ}_{#1}}}
\nc{\BRD}[1]{{{\mathbf{R}}^{|D|}_{#1}}}
\nc{\BRP}[1]{{{\mathbf{R}}^{1}_{#1}}}
\nc{\BRTP}[1]{{{\mathbf{\tilde{R}}}{}^{1}_{#1}}}
\nc{\BS}{{\mathbf{S}}}
\nc{\BT}{{\mathbf{T}}}
\nc{\BMS}{{{\mathbf{M}}^{{s}}}}
\nc{\BMSS}{{{\mathbf{M}}^{{ss}}}}
\nc{\BMZ}{{\mathbf{M}^{\circ}}}
\nc{\BCL}{{\mathbf{L}}}

\nc{\PCC}{{{}^\perp\CC}}

\nc{\Cl}{{\mathsf{Cliff}}}
\nc{\Clev}{{\mathop{\mathsf{Cliff}}^{\circ}}}

\nc{\FA}{{\mathfrak{A}}}
\nc{\FB}{{\mathfrak{B}}}
\nc{\FI}{{\mathfrak{I}}}
\nc{\FZ}{{\mathfrak{Z}}}

\nc{\TFA}{{\tilde{\mathfrak{A}}}}
\nc{\TFB}{{\tilde{\mathfrak{B}}}}

\nc{\fa}{{\mathfrak{a}}}
\nc{\fg}{{\mathfrak{g}}}
\nc{\fp}{{\mathfrak{p}}}
\nc{\FD}{{\mathfrak{D}}}
\nc{\FE}{{\mathfrak{E}}}
\nc{\FL}{{\mathfrak{L}}}
\nc{\FM}{{\mathfrak{M}}}
\nc{\FR}{{\mathfrak{R}}}
\nc{\FS}{{\mathsf{S}}}

\nc{\sfc}{{\mathsf{c}}}
\nc{\sfch}{{\mathsf{ch}}}
\nc{\sfh}{{\mathsf{h}}}

\nc{\SK}{{\mathsf{K}}}
\nc{\SO}{{\mathsf{O}}}
\nc{\SQ}{{\mathsf{Q}}}
\nc{\SPV}{{\mathsf{S}^+\mathsf{V}}}
\nc{\SMV}{{\mathsf{S}^-\mathsf{V}}}
\nc{\SPMV}{{\mathsf{S}^\pm\mathsf{V}}}
\nc{\SX}{{S_X}}
\nc{\SY}{{S_Y}}
\nc{\phipsi}{{q}}
\nc{\eps}{\varepsilon}

\nc{\pim}{{\pi_-}}
\nc{\pip}{{\pi_+}}

\nc{\BE}{{\overline{\CE}}}
\nc{\TE}{{\tilde{\CE}}}
\nc{\TQ}{{\tilde{Q}}}
\nc{\TCF}{{\tilde{\CF}}}
\nc{\TCG}{{\tilde{\CG}}}
\nc{\TCL}{{\tilde{\CL}}}
\nc{\TF}{{\tilde{F}}}
\nc{\TW}{{\tilde{W}}}
\nc{\TCB}{{\widetilde{\CB}}}
\nc{\TCC}{{\tilde{\CC}}}
\nc{\TCX}{{\tilde{\CX}}}
\nc{\TCY}{{\tilde{\CY}}}
\nc{\TPhi}{{\tilde{\Phi}}}
\nc{\OPhi}{{\bar{\Phi}}}
\nc{\txi}{{\tilde{\xi}}}
\nc{\tp}{{\tilde{p}}}
\nc{\tq}{{\tilde{q}}}
\nc{\tzeta}{{\tilde{\zeta}}}
\nc{\tpi}{{\tilde{\pi}}}
\nc{\Tsi}{{\tilde{\Sigma}}}

\nc{\HCB}{{\widehat{\CB}}}
\nc{\HE}{{\widehat{\CE}}}
\nc{\HS}{{\widehat{S}}}
\nc{\HX}{{\hat{X}}}
\nc{\hxi}{{\hat{\xi}}}

\nc{\UH}{{\mathcal{H}}}

\nc{\TM}{{\widetilde{M}}}
\nc{\TCM}{{\widetilde{\CM}}}
\nc{\TS}{{\widetilde{S}}}
\nc{\TU}{{\widetilde{U}}}
\nc{\TX}{{\widetilde{X}}}
\nc{\TY}{{\widetilde{Y}}}
\nc{\TYO}{{{\widetilde{Y}}^\circ}}
\nc{\barf}{{\bar{f}}}
\nc{\te}{{\tilde{e}}{}}
\nc{\tf}{{\tilde{f}}}
\nc{\tg}{{\tilde{g}}}
\nc{\ti}{{\tilde{\imath}}}
\nc{\tj}{{\tilde{\jmath}}}
\nc{\ty}{{\tilde{y}}}
\nc{\tphi}{{\tilde{\phi}}}
\nc{\hf}{{\hat{f}}}

\nc{\urho}{{\underline{\rho}}}

\nc{\LRA}{\Leftrightarrow}
\nc{\RA}{\Rightarrow}
\nc{\lotimes}{\mathbin{\mathop{\otimes}\limits^{\mathbb{L}}}}
\nc{\CEnd}{\mathop{\mathcal{E}\mathit{nd}}\nolimits}
\nc{\CExt}{\mathop{\mathcal{E}\mathit{xt}}\nolimits}
\nc{\CHom}{\mathop{\mathcal{H}\mathit{om}}\nolimits}
\nc{\RH}{\mathop{{\mathsf{R}}\Gamma}\nolimits}
\nc{\RGamma}{\mathop{{\mathsf{R}}\Gamma}\nolimits}
\nc{\RHom}{\mathop{\mathsf{RHom}}\nolimits}
\nc{\RCHom}{\mathop{\mathsf{R}\mathcal{H}\mathit{om}}\nolimits}
\nc{\RG}{\mathop{\mathsf{R\Gamma}}\nolimits}
\nc{\Hom}{\mathop{\mathsf{Hom}}\nolimits}
\nc{\Ext}{\mathop{\mathsf{Ext}}\nolimits}
\nc{\End}{\mathop{\mathsf{End}}\nolimits}
\nc{\Tor}{\mathop{\mathsf{Tor}}\nolimits}
\nc{\Tordim}{\mathop{\mathsf{Tor}\text{\rm-}\mathsf{dim}}\nolimits}
\nc{\Hilb}{\mathop{\mathsf{Hilb}}\nolimits}
\nc{\Spec}{\mathop{\mathsf{Spec}}\nolimits}
\nc{\Proj}{\mathop{\mathsf{Proj}}\nolimits}
\nc{\Pic}{\mathop{\mathsf{Pic}}\nolimits}

\nc{\Tw}{\mathop{\mathsf{Tw}}\nolimits}
\nc{\Ker}{\mathop{\mathsf{Ker}}\nolimits}
\nc{\Coker}{\mathop{\mathsf{Coker}}\nolimits}
\nc{\codim}{\mathop{\mathsf{codim}}\nolimits}
\nc{\sing}{{\mathsf{sing}}}
\nc{\supp}{\mathop{\mathsf{supp}}}
\nc{\perf}{{\mathsf{perf}}}
\nc{\rank}{\mathop{\mathsf{rank}}}
\nc{\Pf}{{\mathsf{Pf}}}
\nc{\Gr}{{\mathsf{Gr}}}
\nc{\OGr}{{\mathsf{OGr}}}
\nc{\Flag}{{\mathsf{Fl}}}
\nc{\Kosz}{{\mathsf{Kosz}}}
\nc{\LGr}{{\mathsf{LGr}}}
\nc{\GTGr}{{\mathsf{G_2Gr}}}
\nc{\GTF}{{\mathsf{G_2F}}}
\nc{\OF}{{\mathsf{OF}}}
\nc{\Fl}{{\mathsf{Fl}}}
\nc{\Bl}{{\mathsf{Bl}}}
\nc{\GL}{{\mathsf{GL}}}
\nc{\PGL}{{\mathsf{PGL}}}
\nc{\SL}{{\mathsf{SL}}}
\nc{\SP}{{\mathsf{Sp}}}
\nc{\Spin}{{\mathsf{Spin}}}
\nc{\Tot}{{\mathsf{Tot}}}
\nc{\ev}{{\mathsf{ev}}}
\nc{\od}{{\mathsf{odd}}}
\nc{\coev}{{\mathsf{coev}}}
\nc{\id}{{\mathsf{id}}}
\nc{\opp}{{\mathsf{opp}}}
\nc{\PS}{{{\PP^3}}}
\nc{\Qu}{{{Q^3}}}
\nc{\tdim}{\mathop{\Tor\dim}}
\nc{\ecart}{{\fbox{$\scriptstyle\mathsf{EC}$}}}
\nc{\ad}{{\mathop{\mathsf ad}}}
\nc{\gr}{{\mathop{\mathsf gr}}}
\nc{\qgr}{{\mathop{\mathsf qgr}}}
\nc{\tor}{{\mathop{\mathsf tor}}}
\rnc{\mod}{{\mathop{\mathsf mod}}}
\nc{\Mod}{{\mathop{\mathsf Mod}}}
\nc{\Coh}{{\mathop{\mathsf Coh}}}
\nc{\Ab}{{\mathop{\mathcal{A}\mathit{b}}}}
\nc{\QCoh}{{\mathop{\mathsf QCoh}}}

\nc{\AAV}{{\mathcal{AAV}}}

\nc{\Rep}{{\mathsf{Rep}}}

\nc{\Cubics}{{{\mathcal{S}}_3}}
\nc{\VFT}{{{\mathcal{S}}_{14}}}
\nc{\VFTE}{{{\mathcal{N}}_{\mathrm{reg,sm}}}}
\nc{\MX}{{\CM_X}}
\nc{\MY}{{\CM_Y}}
\nc{\MYE}{{\CM_{Y,\CE}}}
\nc{\Yd}{{Y_d}}
\nc{\Yfive}{{Y_5}}
\nc{\Xg}{{X_{2g-2}}}
\nc{\Xtt}{{X_{22}}}
\nc{\Xst}{{X_{16}}}
\nc{\Xtw}{{X_{12}}}
\nc{\Xe}{{X_{8}}}
\nc{\Xf}{{X_{4}}}

\nc{\git}{{/\!\!/\!{}_\chi}}

\theoremstyle{plain}

\newtheorem{theorem}{Theorem}[section]

\newtheorem{lemma}[theorem]{Lemma}
\newtheorem{proposition}[theorem]{Proposition}
\newtheorem{corollary}[theorem]{Corollary}

\theoremstyle{definition}

\theoremstyle{remark}

\newtheorem{remark}[theorem]{Remark}

\newenvironment{proof}{\noindent{\sf Proof:}}{\qed\medskip}

\title[Scheme of lines on a family of quadrics: geometry and derived category]%
{Scheme of lines on a family of 2-dimensional quadrics:\\[1em]geometry and derived category}
\author{Alexander Kuznetsov}
\address{\sloppy
\parbox{0.9\textwidth}{
Algebra Section, Steklov Mathematical Institute,
8 Gubkin str., Moscow 119991 Russia
\hfill\\[5pt]
The Poncelet Laboratory, Independent University of Moscow
\hfill
}\bigskip}
\email{akuznet@@mi.ras.ru}
\date{}
\thanks{I was partially supported by
RFFI grants
08-01-00297,
09-01-12170,
10-01-93110,
10-01-93113,
NSh-4713.2010.1,
and
Russian Presidential grant for young scientists MD-2712.2009.1.}

\begin{document}

\maketitle

\begin{abstract}
Given a generic family $Q$ of 2-dimensional quadrics over a smooth 3-dimensional base $Y$
we consider the relative Fano scheme $M$ of lines of it. The scheme $M$ has a structure of a generically
conic bundle $M \to X$ over a double covering $X \to Y$ ramified in the degeneration locus of $Q \to Y$.
The double covering $X \to Y$ is singular in a finite number of points (corresponding to the points $y \in Y$
such that the quadric $Q_y$ degenerates to a union of two planes), the fibers of $M$ over such points are
unions of two planes intersecting in a point.

The main result of the paper is a construction of a semiorthogonal decomposition for the derived category
of coherent sheaves on $M$. This decomposition has three components, the first is the derived category of a small resolution $X^+$
of singularities of the double covering $X \to Y$, the second is a twisted resolution of singularities of $X$
(given by the sheaf of even parts of Clifford algebras on $Y$), and the third is generated by a completely
orthogonal exceptional collection.
\end{abstract}

\section{Introduction}

The subject of this note is a description of the structure of the derived category
of coherent sheaves on the relative scheme of lines for a family of 2-dimensional quadrics.
We had two motivations for investigation of this category --- first of all it has
an interesting structure and exhibits some interesting features. For example,
it combines the minimal resolution of singularities and the twisted resolution
of singularities of a certain double covering of the base of the family.

The second, and the most important motivation, comes from investigation
of the derived categories of some special double covers of $\PP^3$
and their relation to Enriques surfaces. See the companion
paper~\cite{IK} for details.

The precise formulation of the main result of the paper is the following.
Consider a family of 2-dimensional quadrics $q:Q \to Y$.
This means that we are given a projectivization of a rank 4 vector bundle
$\CV$ on $Y$ and a divisor $Q \subset \PP_{Y}(\CV)$ of relative degree 2
which is flat over $Y$. Such divisor is given by a line subbundle $\CL \subset S^2\CV^\vee$.

Given this we consider {\em the relative Fano scheme of lines of $Q$ over $Y$}.
By definition this is the zero locus on the relative Grassmannian $\Gr_{Y}(2,\CV)$
of the global section
$$
s \in \Gamma(\Gr_Y(2,\CV),\CL^\vee\otimes S^2\CU^\vee),
$$
where $\CU \subset \CV$ is the tautological subbundle on the Grassmannian.
We denote this relative Fano scheme by $M$.
The fibers $M_y$ of the projection $\rho:M \to Y$ have the following structure
\begin{itemize}
\item $M_y$ is a disjoint union of two smooth conics, if the quadric $Q_y$ is smooth;
\item $M_y$ is a single smooth conic (with a nonreduced scheme structure), if the quadric $Q_y$ has corank~1;
\item $M_y$ is a union of two planes intersecting in a point, if the quadric $Q_y$ has corank 2;
\item $M_y$ is a single plane (with a nonreduced scheme structure), if the quadric $Q_y$ has corank 3.
\end{itemize}
From now on we assume that the family $Q \to Y$ is sufficiently generic, that is
generic fiber is smooth, and the codimension of the locus $D_r \subset Y$ of quadrics
of corank $r$ equals $r(r+1)/2$.
In this case we have the Stein factorization for the morphism $\rho:M \to Y$:
$$
\xymatrix{
M \ar[rr]^\mu \ar[dr]_\rho && X \ar[dl]^f \\ & Y
}
$$
where $f:X \to Y$ is the double covering ramified in the divisor $D_1$, and
$\mu:M \to X$ is generically a conic bundle.

The main result of this paper is a description of the derived category of $M$ when $\dim Y = 3$.
Then the above genericity assumptions imply that $D_3 = \emptyset$
and $D_2$ consists of a finite number $N$ of isolated points $y_1,\dots,y_N$.
Additionally we assume that $D_1$ has an ordinary double point
(an ODP or a node for short) in each of $y_i$. The last assumption
is equivalent to smoothness of $M$ if $Q$ is smooth.

To state the answer we need the following ingredients.
First, consider the sheaf of even parts of Clifford algebras $\CB_0$ on $Y$
associated with the family $Q \to Y$ (see~\cite{K08a} for details).
As an $\CO$-module it is given by
$$
\CB_0 = \CO_Y \oplus \Lambda^2\CV\otimes\CL \oplus \Lambda^4\CV\otimes\CL^2
$$
with the Clifford multiplication. If $Q$ is smooth then the category
$\D^b(Y,\CB_0)$ is also smooth and can be thought of as a twisted noncommutative
resolution of the double covering $X$.

Second, note that $X$ has $N$ isolated ordinary double points over $D_2$,
so being 3-dimensional it has $2^N$ small resolutions of singularities in the category
of Moishezon varieties. To fix one of these resolutions we should choose for each
point $y_i \in D_2$ one of the planes in the corank 2 quadric $Q_{y_i}$, or equivalently
one of the planes in the fiber $M_{y_i}$ of $M$ over $Y$.

Let us pick one of these resolutions and denote it by $\sigma_+:X^+ \to X$.
Let us denote the planes in $M_{y_i}$ corresponding to this choice by $\Sigma_i^+$,
and the complementary planes by $\Sigma_i^-$, so that
$M_{y_i} = \Sigma_i^+ \cup \Sigma_i^-$.

The main result of this paper is the following

\begin{theorem}\label{mt}
Assume that $Q \to Y$ is a family of quadrics,
$Y$ and $Q$ are smooth, $\dim Y = 3$, and the degeneration locus $D_1$
has a finite number of ordinary double points $\{y_1,\dots,y_N\} = D_2$.
Then the relative Fano scheme $M$ of lines of $Q$ over $Y$ is smooth and
there is a semiorthogonal decomposition
$$
\D^b(M) = \langle \D^b(X^+), \D^b(Y,\CB_0), \{ \CO_{\Sigma_i^+} \}_{i=1}^N \rangle.
$$
Here the third component is a completely orthogonal exceptional collection.
\end{theorem}

The proof goes as follows.
In section~\ref{s-pl} we prove smoothness of $M$ and investigate the local structure
of $M$ around the planes $\Sigma_i^\pm$. In particular, we check that the sheaves
$\CO_{\Sigma_i^+}$ form a completely orthogonal exceptional collection in $\D^b(M)$.
In section~\ref{s-ca} we recall some facts about the sheaf of even parts of Clifford
algebras $\CB_0$ and construct a fully faithful embedding $\D^b(Y,\CB_0) \to \D^b(M)$.
In section~\ref{s-fl} we show that there is a birational transformation of $M$,
a flip in $N$ planes $\Sigma_i^+$, transforming it into a $\PP^1$-fibration
$\mu_+:M^+ \to X^+$ over a small resolution $X^+ \to X$.
This gives an identification of the orthogonal to the collection
$\{ \CO_{\Sigma_i^+} \}_{i=1}^N$ in $\D^b(M)$ with $\D^b(M^+)$.
In section~\ref{s-mp} we construct a fully faithful embedding $\D^b(X^+) \to \D^b(M^+)$
and identify the complement with $\D^b(Y,\CB_0)$.

In the last section~\ref{s-f} we discuss another way of proving Theorem~\ref{mt}
and suggest some further directions of investigation.

{\bf Acknowledgement:} I would like to thank L.Katzarkov, D.Orlov, and Yu.Prokhorov for helpful discussions.

\section{Geometry of $M$}\label{s-pl}

For each quadric $Q_y$ in the family $Q \to Y$ denote by $K_y \subset \CV_y$ the kernel
of the corresponding quadratic form (thus $\PP(K_y)$ is the singular locus of $Q_y$).
Note that the differential of the section $s \in \Gamma(Y,\CL^{-1}\otimes S^2\CV^\vee)$
at $y$ gives a linear map $T_yY \to S^2\CV_y^\vee$. Composing it with the natural projection
$S^2\CV_y^\vee \to S^2K_y^\vee$ we obtain a map
$$
\kappa_y:T_yY \to S^2K_y^\vee.
$$
In term of these maps one can check the smoothness of $Q$ and $M$.

\begin{proposition}\label{qmsm}
Assume that $Y$ is smooth. Then
\begin{enumerate}
\item $Q$ is smooth if and only if for any $y \in Y$ and any subspace $K \subset K_y$ with $\dim K \le 1$,
the composition $\xymatrix@1{T_yY \ar[r] & S^2K_y^\vee \ar[r] & S^2K^\vee}$ is surjective;
\item $M$ is smooth if and only if for any $y \in Y$ and any embedding $K \to K_y$ with $\dim K \le 2$,
the composition $\xymatrix@1{T_yY \ar[r] & S^2K_y^\vee \ar[r] & S^2K^\vee}$ is surjective.
\end{enumerate}
\end{proposition}
\begin{proof}
The question is local so we can assume that $\CV$ is a trivial bundle, $\CV \cong V\otimes\CO_Y$.
Then the result is a simple local calculation.
\end{proof}

\begin{remark}
This result generalizes to arbitrary relative isotropic Grassmannians of families
of quadrics of arbitrary dimension. The smoothness of the Grassmannian of $k$-dimensional
subspaces is equivalent to the surjectivity of the corresponding map for all $K$ with $\dim K \le k$.
\end{remark}

\begin{corollary}\label{msm}
Assume that $Y$ is smooth and $D_3 = \emptyset$.
Then $M$ is smooth if and only if $Q$ is smooth and for any $y \in D_2$ the map
$\kappa_y:T_yY \to S^2K_y^\vee$ is surjective.
\end{corollary}

Another consequence of smoothness of $Q$ is smoothness of $D_1 \setminus D_2$.
On the other hand, the points of $D_2$ are always singular on $D_1$.
In fact they are ordinary double points if $M$ is smooth.

\begin{lemma}
Assume that $Q$ is smooth and $\dim Y = 3$. Then
$M$ is smooth if and only if $D_2$ is a finite number of points and
any point of $D_2$ is an ordinary double point of $D_1$.
\end{lemma}
\begin{proof}
Take any $y \in D_2$, so that $\dim K_y = 2$.
The map $T_yY \to S^2K_y^\vee$ can be thought of as a net of quadrics on $K_y$
parameterized by $T_yY$. Its degeneration locus in $\PP(T_yY)$ is a conic,
either nondegenerate (if the map $T_yY \to S^2K_y^\vee$ is surjective) or
singular (since the kernel of the map lies in the singular locus).
But on the other hand, this degeneration locus is the base of the tangent
cone to $D_1$ at $y_i$. So, if $y_i$ is an ODP of $D_1$, the conic should be nondegenerate,
hence the map should be surjective.
\end{proof}

\begin{remark}\label{kiso}
Note also that if the map $\kappa_y$ for $y\in D_2$ is surjective,
then it is an isomorphism (since $\dim T_yY = \dim S^2K_y^\vee = 3$).
\end{remark}

For each point $y_i \in D_2$ we have $Q_{y_i} = \PP(W_i^+) \cup \PP(W_i^-)$,
both $W_i^+$ and $W_i^-$ being a 3-dimensional subspaces in $\CV_{y_i}$, the fiber
of $\CV$ over $y_i$. These planes intersect along the line $\PP(K_{y_i})$.
We put
$$
W_i^0 = K_{y_i} = W_i^+ \cap W_i^-.
$$
Consequently, $M_{y_i} = \Gr(2,W_i^+) \cup \Gr(2,W_i^-) = \Sigma_i^+ \cup \Sigma_i^-$,
both $\Sigma_i^+$ and $\Sigma_i^-$ are planes. These planes intersect in a point $P_i = \Gr(2,W_i^0)$.
Choose arbitrary point $y = y_i \in D_2$ and one of the planes $\Sigma = \Sigma_i^\pm$.

\begin{proposition}
If a point $y \in D_2$ is an ODP of $D_1$ then
$\CN_{\Sigma/M} \cong \CO_\Sigma(-1) \oplus \CO_\Sigma(-1)$.
\end{proposition}
\begin{proof}
Choosing a local trivialization of the bundle $\CV$ we obtain an isomorphism
$$
\CN_{\Sigma/\Gr_Y(2,\CV)} \cong
\CN_{\Sigma/\Gr(2,\CV_y)} \oplus T_yY\otimes\CO_\Sigma \cong
\CU^\vee \oplus T_yY \otimes \CO_Y.
$$
On the other hand,
$\CN_{M/\Gr_Y(2,\CV)} \cong S^2\CU^\vee\otimes\CL^\vee$.
Hence the standard exact sequence
$$
0 \to \CN_{\Sigma/M} \to \CN_{\Sigma/\Gr_Y(2,\CV)} \to (\CN_{M/\Gr_Y(2,\CV)})_{|\Sigma} \to 0
$$
gives
$$
0 \to \CN_{\Sigma/M} \to \CU^\vee_{|\Sigma} \oplus T_yY\otimes\CO_{\Sigma} \to S^2\CU^\vee_{|\Sigma} \to 0.
$$
Since $\Sigma = \Gr(2,W)$, $W\subset \CV_y$, the cohomology exact sequence looks like
$$
0 \to H^0(\Sigma,\CN_{\Sigma/M}) \to W^* \oplus T_yY \to S^2W^* \to H^1(\Sigma,\CN_{\Sigma/M}) \to 0.
$$
Consider the map $W^* \oplus T_yY \to S^2W^*$.
Its first component is the multiplication by the equation of the line $K_y = W^0 \subset W$.
Hence the sequence can be rewritten as
$$
0 \to H^0(\Sigma,\CN_{\Sigma/M}) \to T_yY \to S^2K_y^\vee \to H^1(\Sigma,\CN_{\Sigma/M}) \to 0.
$$
The middle map here is just the map $\kappa_y$, hence by Remark~\ref{kiso} it is an isomorphism.
Thus the bundle $\CN_{\Sigma/M}$ is acyclic, hence it is isomorphic to $\CO_\Sigma(-1) \oplus \CO_\Sigma(-1)$.
\end{proof}

From now on we assume that every point $y_i \in D_2$ is an ODP of $D_1$.

\begin{corollary}\label{omms}
We have $(\omega_{M})_{|\Sigma^\pm_i} \cong \CO_{\Sigma^\pm_i}(-1)$.
\end{corollary}
\begin{proof}
By adjunction formula
$\CO_{\Sigma}(-3) \cong
\omega_{\Sigma} \cong
\omega_{M|\Sigma} \otimes \det \CN_{\Sigma/M} \cong
\omega_{M|\Sigma} \otimes \CO_{\Sigma}(-2)$,
hence the claim.
\end{proof}

The most important corollary is the following

\begin{corollary}
The structure sheaf $\CO_\Sigma \in \D^b(M)$ is exceptional.
\end{corollary}
\begin{proof}
We have an isomorphism $\CExt^t(\CO_\Sigma,\CO_\Sigma) \cong \Lambda^t\CN_{\Sigma/M} \cong \Lambda^t(\CO_\Sigma(-1) \oplus \CO_\Sigma(-1))$.
Note that for $t = 1$, and $t = 2$ this sheaf on $\Sigma = \PP^2$ is acyclic.
Hence $\Ext^\bullet(\CO_\Sigma,\CO_\Sigma) \cong H^\bullet(\Sigma,\CHom(\CO_\Sigma,\CO_\Sigma)) \cong H^\bullet(\Sigma,\CO_\Sigma)$
implies exceptionality of $\CO_\Sigma$.
\end{proof}

Another simple observation is that $\Sigma_i^\pm$ with different $i$ are completely orthogonal.

\begin{lemma}
If $i \ne j$ then $\Ext^\bullet(\CO_{\Sigma_i^\pm},\CO_{\Sigma_j^\pm}) = 0$.
\end{lemma}
\begin{proof}
The planes $\Sigma_i^\pm$ and $\Sigma_j^\pm$ are contained in fibers of $M$ over different points $y_i,y_j \in Y$,
hence there is no local $\Ext$'s between their structure sheaves. Hence global $\Ext$'s also vanish.
\end{proof}

Thus choosing one plane for each $y_i$ we obtain a completely orthogonal exceptional collection

\begin{corollary}
The collection $\{ \CO_{\Sigma_i^+} \}_{i=1}^N$ is a completely orthogonal exceptional collection in $\D^b(M)$.
\end{corollary}

\section{The Clifford algebra}\label{s-ca}

For the precise definition and basic results about the sheaves
of even parts of Clifford algebras, see~\cite{K08a}. Here we remind
some of their properties.

Recall that the besides the sheaf of algebras $\CB_0 = \CO_Y \oplus \Lambda^2\CV\otimes\CL \oplus \Lambda^4\CV\otimes\CL^2$ on $Y$,
we also have a natural sequence of sheaves of $\CB_0$-modules,
the first of them is the odd part of the sheaf of Clifford algebras $\CB_1$,
which as a sheaf of $\CO_Y$-modules is given by
$$
\CB_1 = \CV \oplus \Lambda^3\CV\otimes\CL
$$
and with the action of $\CB_0$ given by the Clifford multiplication.
The other sheaves $\CB_k$ in the sequence are obtained from $\CB_0$ and $\CB_1$
by an appropriate twist
$$
\CB_{-2k} = \CB_0 \otimes \CL^k,
\qquad
\CB_{1-2k} = \CB_1 \otimes \CL^k.
$$
This sequence can be thought of as a sequence of powers of a line bundle.
In particular, the functors $-\otimes_{\CB_0}\CB_l$ and $\CHom_{\CB_0}(\CB_l,-)$
are exact and we have
\begin{equation}\label{blbk}
\CB_k\otimes_{\CB_0}\CB_l \cong \CB_{k+l},\qquad
\RCHom_{\CB_0}(\CB_l,\CB_k) \cong \CB_{k-l}.
\end{equation}

Let $\alpha$ denote the embedding $M \to \Gr_Y(2,\CV)$.
Let $g$ denote $c_1(\CU^\vee)$, the positive generator of
the relative Picard group $\Pic(\Gr_Y(2,\CV)/Y)$.
Since $M$ is the zero locus of $s \in \Gamma(\Gr_Y(2,V),\CL^\vee\otimes S^2\CU^\vee)$
we have the Koszul resolution for its structure sheaf
\begin{equation}\label{om}
0 \to \CL^3(-3g) \to \CL^2\otimes S^2\CU(-g) \to \CL\otimes S^2\CU \to \CO \to \alpha_*\CO_M \to 0.
\end{equation}
Now we will show that $M$ also comes with a sequence of naturally defined $\CB_0$-modules.
To unburden the notation we denote the pullbacks of the sheaves $\CB_k$ to $\Gr_Y(2,\CV)$ by the same letters.
For each $k \in \ZZ$ consider the morphism $\CU\otimes\CB_{k-1} \to \CB_k$ of sheaves of $\CB_0$-modules
on $\Gr_Y(2,\CV)$ induced by the embedding $\CU\subset \CV$ and the Clifford multiplication $\CV\otimes\CB_k \to \CB_{k+1}$.

\begin{proposition}
There are isomorphisms $\Coker(\CU\otimes\CB_{k-1} \to \CB_k) \cong \alpha_*\CS_k$,
where
\begin{equation}\label{sodd}
\CS_{2k+1} = (\CV/\CU)\otimes\CL^{-k},
\end{equation}
and there is an exact sequence
\begin{equation}\label{seven}
0 \to \CL^{-k} \to \CS_{2k} \to \det\CV\otimes\CL^{1-k}(g) \to 0.
\end{equation}
Moreover, the sheaves $\CS_k$ have a structure of $\CB_0$-modules such that
\begin{equation}\label{blsk}
\CS_k\otimes_{\CB_0}\CB_l \cong \CS_{k+l},
\qquad
\RCHom_{\CB_0}(\CB_l,\CS_k) \cong \CS_{k-l}.
\end{equation}
Finally, for each $k$ there is an exact sequence
\begin{equation}\label{alsk}
0 \to \CO(-2g)\otimes\CB_{k-4} \to \CU(-g)\otimes\CB_{k-3} \to \CU\otimes\CB_{k-1} \to \CB_k \to \alpha_*\CS_k \to 0.
\end{equation}
\end{proposition}
\begin{proof}
First let us check that the cokernels are supported on $M$ scheme-theoretically.
For this we note that the composition of the maps
$S^2\CU\otimes\CB_{k-2} \to \CU\otimes\CB_{k-1} \to \CB_k$
(both of which are induced by the action of $\CV$ on $\CB$)
coincides with the map
$S^2\CU\otimes\CB_{k-2} \cong S^2\CU\otimes\CB_k\otimes\CL \to \CB_k$
induced by the section $s$
defining the family $Q$ (this follows from the definition of the Clifford multiplication).
It follows that the cokernel is a quotient of $\alpha_*\alpha^*\CB_k$, hence it can be written
as $\alpha_*\CS_k$, where $\CS_k$ is a sheaf of $\CB_0$-modules on $M$.
Note also that the formulas~\eqref{blsk} follow from the definition of $\CS_k$ combined with equations~\eqref{blbk}
and exactness of functors $-\otimes_{\CB_0}\CB_l$ and $\CHom_{\CB_0}(\CB_l,-)$.
So, it remains to verify~\eqref{sodd}, \eqref{seven} and~\eqref{alsk}.

For this consider the maps $\CB_{k-4} \to \CB_{k-3}\otimes\CU^\vee$ obtained by the partial dualization
from the maps $\CU\otimes\CB_{k-4} \to \CB_{k-3}$. Also consider the composition
$\CU\otimes\CB_{k-3} \to \CV\otimes\CB_{k-3} \to \CV^\vee\otimes\CB_{k-1} \to \CU^\vee\otimes\CB_{k-1}$,
where the middle map is induced by the double action of $\CV$ on $\CB$.
Finally, after appropriate twistings and identification $\CU^\vee(-g) \cong \CU$, we compose a sequence
$$
0 \to \CO(-2g)\otimes\CB_{k-4} \to \CU(-g)\otimes\CB_{k-3} \to \CU\otimes\CB_{k-1} \to \CB_k \to 0.
$$
It is easy to check that the compositions of the arrows are zero, so the constructed sequence of maps is a complex.
Note that each term of the complex is naturally filtered. Consider the spectral sequence of the filtered complex
in case $k = 0$. The first term looks like
$$
\xymatrix@R=5pt{
\CO(-2g) \otimes \Lambda^4\CV \otimes \CL^4 \\
\CO(-2g) \otimes \Lambda^2\CV \otimes \CL^3 \ar[r] & \CU(-g) \otimes \Lambda^3\CV \otimes \CL^3 \\
\CO(-2g) \otimes \CL^2 \ar[r] & \CU(-g) \otimes \CV \otimes \CL^2 \ar[r] & \CU \otimes \Lambda^3\CV \otimes \CL^2 \ar[r] & \Lambda^4\CV \otimes \CL^2 \\
&& \CU \otimes \CV \otimes \CL \ar[r] & \Lambda^2\CV \otimes \CL \\
&&& \CO
}
$$
The rows are natural complexes with maps corresponding to the wedge multiplication.
Their cohomology are easy to compute, so it is not difficult to see that the second term looks like
$$
\xymatrix@R=5pt{
\det\CV\otimes\CL^4(-2g) \ar[dr] \\
\CL^3(-3g) \ar[dr] & \det\CV\otimes\CL^3\otimes S^2\CU \ar[dr] \\
& \CL^2\otimes S^2\CU(-g) \ar[dr] & \det\CV\otimes\CL^2\otimes S^2\CU(g) \ar[dr] \\
&& \CL\otimes S^2\CU \ar[dr] & \det\CV\otimes\CL\otimes\CO(g) \\
&&& \CO
}
$$
The arrows here are induced by $s$. So, it is easy to see that the bottom chain is the Koszul complex of $s$,
while the top chain is the same complex twisted by $\det\CV\otimes\CL(g)$.
Hence the spectral sequence degenerates in the third term and shows that the cohomology of the above complex
is supported in degree zero and is an extension of $\alpha_*\alpha^*(\det\CV\otimes\CL(g))$ by $\alpha_*\alpha^*\CO$.
Since we already know that it is supported on~$M$, we conclude that it can be written as $\alpha_*\CS_0$,
where $\CS_0$ is an extension of $\det\CV\otimes\CL(g)$ by $\CO$ on $M$. This gives~\eqref{seven} and~\eqref{alsk} for $\CS_0$.

Analogously, consider the complex for $k = 1$.
The first term of the spectral sequence looks like
$$
\xymatrix@R=5pt{
\CO(-2g)\otimes\Lambda^3\CV\otimes\CL^3 \ar[r] & \CU(-g)\otimes\Lambda^4\CV\otimes\CL^3 \\
\CO(-2g)\otimes\CV\otimes\CL^2 \ar[r] & \CU(-g)\otimes\Lambda^2\CV\otimes\CL^2 \ar[r] & \CU\otimes\Lambda^4\CV\otimes\CL^2 \\
& \CU(-g)\otimes\CL \ar[r] & \CU\otimes\Lambda^2\CV\otimes\CL \ar[r] & \Lambda^3\CV\otimes\CL \\
&& \CU \ar[r] & \CV
}
$$
The maps are induced by the wedge multiplication, so one can check that the second term looks like
$$
\xymatrix@R=5pt{
\CV/\CU\otimes\CL^3(-3g) \ar[dr] \\
& \CV/\CU \otimes\CL^2\otimes S^2\CU(-g) \ar[dr] \\
&& \CV/\CU \otimes \CL\otimes S^2\CU \ar[dr] \\
&&& \CV/\CU
}
$$
The maps are induced by $s$, so it is the Koszul complex of $s$ tensored with $\CV/\CU$, hence $\CS_1 \cong \CV/\CU$.
This gives~\eqref{sodd} and~\eqref{alsk} for $\CS_1$.
For other $\CS_k$ we deduce~\eqref{sodd}, \eqref{seven} and~\eqref{alsk} by a suitable twist.
%
\end{proof}

Applying the functor $\rho_*$ to the resolutions~\eqref{alsk} twisted by $\CO(-g)$ we deduce the following

\begin{corollary}\label{rhosk}
We have $\rho_*(\CS_k) \cong \CB_k$, $\rho_*(\CS_k(-g)) = 0$.
\end{corollary}

\begin{corollary}\label{exts0}
The extension in~\eqref{seven} is nontrivial.
\end{corollary}
\begin{proof}
Assume that $\CS_0 \cong \CO \oplus \det\CV\otimes\CL(g)$. Then
$$
\rho_*(\CS_0(-g)) \cong
\rho_*((\CO \oplus \det\CV\otimes\CL(g))(-g)) \cong
\rho_*(\CO(-g) \oplus \det\CV\otimes\CL).
$$
Using~\eqref{om} it is easy to see that $\rho_*\CO_M = \CO_Y \oplus \det\CV\otimes\CL^2$ and $\rho_*(\CO(-g)) = \det\CV\otimes\CL \oplus (\det\CV)^2\otimes\CL^3$,
so the RHS is nontrivial, which contradicts~\ref{rhosk}.
\end{proof}

\begin{corollary}\label{rhosd}
We have $\rho_*(\CS_k^\vee) = 0$.
\end{corollary}
\begin{proof}
Indeed, since $\CS_k$ is of rank 2 and $\det\CS_k = \det\CV\otimes\CL^{1-k}(g)$,
we have $\CS_k^\vee \cong \CS_k(-g)\otimes\det\CV^\vee\otimes\CL^{k-1}$, hence its
pushforward is a twist of $\rho_*(\CS_k(-g))$ which is zero.
\end{proof}

Another consequence is the following

\begin{corollary}\label{rhosds}
We have $\rho_*(\CS_l^\vee\otimes\CS_k) \cong\CB_{k-l}$.
\end{corollary}
\begin{proof}
First of all consider the case $l = 0$. Then dualizing~\eqref{seven} we obtain an exact triple
$$
0 \to \det\CV^\vee\otimes\CL^{-1}(-g) \to \CS_0^\vee \to \CO_M \to 0.
$$
Tensoring it by $\CS_k$, pushing forward and using~\ref{rhosk}, we obtain the claim.
Now for arbitrary $l$ the formula follows by tensoring with $\CB_{-l}$ and using~\eqref{blsk}.
\end{proof}

Now we can describe the embedding $\D^b(Y,\CB_0) \to \D^b(M)$.

\begin{theorem}\label{phi}
The functor $\Phi:\D^b(Y,\CB_0) \to \D^b(M)$, $\CF \mapsto \CS_0\otimes_{\CB_0} \rho^*\CF$ is fully faithful.
Moreover,
$$
\Phi(\CB_k) \cong \CS_k.
$$
\end{theorem}
\begin{proof}
First, note that
$$
\Hom(\Phi(\CF),\CG) =
\Hom(\CS_0\otimes_{\CB_0} \rho^*\CF,\CG) \cong
\Hom_{\CB_0}(\rho^*\CF,\CS_0^\vee\otimes_{\CO_M}\CG) \cong
\Hom_{\CB_0}(\CF,\rho_*(\CS_0^\vee\otimes_{\CO_M}\CG)).
$$
Thus the right adjoint functor $\Phi^!:\D^b(M) \to \D^b(Y,\CB_0)$ is given by
$$
\Phi^!(\CG) = \rho_*(\CS_0^\vee\otimes_{\CO_M}\CG)
$$
(the structure of $\CB_0$-module is induced by that of $\CS_0^\vee$). So, to check full faithfulness
it suffices to compute $\Phi^!\circ\Phi$. For this we note that
$$
\Phi^!(\Phi(\CF)) =
\rho_*(\CS_0^\vee\otimes_{\CO_M} \CS_0\otimes_{\CB_0} \rho^*\CF) \cong
\rho_*(\CS_0^\vee\otimes_{\CO_M} \CS_0)\otimes_{\CB_0}\CF \cong
\CB_0 \otimes_{\CB_0} \CF \cong \CF
$$
(we applied~\ref{rhosds}).
Thus $\Phi^!\circ\Phi \cong \id$, so $\Phi$ is fully faithful.
Finally, $\Phi(\CB_k) = \CS_0\otimes_{\CB_0}\CB_k \cong \CS_k$ by~\eqref{blsk}.
\end{proof}

We conclude the section with the following simple calculation.

\begin{lemma}\label{sks}
For each $i$ and each $k$ we have $(\CS_k)_{|\Sigma_i^\pm} \cong \CO_{\Sigma_i^\pm} \oplus \CO_{\Sigma_i^\pm}(1)$.
\end{lemma}
\begin{proof}
Restrict~\eqref{sodd} and~\eqref{seven} to $\Sigma = \Sigma_i^\pm$. Since $\CO(g)$ restricts to $\Sigma$ as $\CO_\Sigma(1)$ we obtain the claim for even $k$.
For odd $k$ we have to describe the restriction of $\CV/\CU$ to $\Sigma$. Since $\Sigma = \Gr(2,W) \subset \Gr(2,\CV_y)$,
we have on $\Sigma$ an exact sequence
$$
0 \to W/\CU \to \CV/\CU \to \CV_y/W \otimes \CO_\Sigma \to 0.
$$
The first term is $\CO_\Sigma(1)$ and the third is $\CO_\Sigma$. Hence $(\CV/\CU)_{|\Sigma} \cong \CO_\Sigma \oplus \CO_\Sigma(1)$.
\end{proof}

\section{The flip}\label{s-fl}

From now on we choose one of the planes $\PP(W_i^\pm) \subset Q_{y_i}$ for each point $y_i$,
say $\PP(W_i^+)$, and the corresponding plane $\Sigma_i^+ = \Gr(2,W_i^+) \subset M$.
Recall that the normal bundles of $\Sigma_i^+$ in $M$ are $\CO(-1) \oplus \CO(-1)$.
Let us apply the composition of flips in all these planes and denote by $M^+$ the resulting
Moishezon variety. More precisely, consider the blowup $\xi:\TM \to M$ of $M$ in the union
of all $\Sigma_i^+$. Then each of the exceptional divisors $E_i = \xi^{-1}(\Sigma_i^+)$
is isomorphic to $\Sigma_i^+\times\PP^1$ and its normal bundle is $\CO(-1,-1)$.
Hence in the category of Moishezon varieties it can be blown down onto a line $L_i \cong \PP^1 \subset M^+$.
Thus we have a diagram
$$
\xymatrix@!C{
& \bigsqcup E_i \ar[dl] \ar[dr] \ar[d] \\
\bigsqcup \Sigma_i^+ \ar[d] & \TM \ar[dl]_\xi \ar[dr]^{\xi^+} & \bigsqcup L_i \ar[d] \\
M && M^+
}
$$

By a result of Bondal and Orlov we have the following

\begin{proposition}[\cite{BO95}]
The functor $\xi_*(\xi^+)^*:\D^b(M^+) \to \D^b(M)$ is fully faithful.
Moreover, there is a semiorthogonal decomposition
$$
\D^b(M) = \langle \xi_*(\xi^+)^*(\D^b(M^+)), \{ \CO_{\Sigma_i^+} \}_{i=1}^N  \rangle.
$$
\end{proposition}

Further, we will need a detailed description of the fibers of $M^+$  over $X$.
Let $x_i = f^{-1}(y_i)$ be the nodal points of $X$ and $X_{sm} = X \setminus \{ y_i \}_{i=1}^N$
be the smooth locus of $X$.

\begin{lemma}
There is a regular morphism $M^+ \to X$ such that the diagram
$$
\xymatrix{
& \TM \ar[dl]_\xi \ar[dr]^{\xi_+} \\ M \ar[dr]_\mu && M^+ \ar[dl] \\ & X
}
$$
commutes. Moreover, over the smooth locus $X_{sm}$ the maps $M \to X$ and $M^+ \to X$ coincide.
Finally, the fiber $M^+_{x_i}$ of $M^+$ over $x_i$ is the blowup $\Tsi_i^-$ of $\Sigma_i^-$ in the points $P_i$
and the line $L_i = \xi_+(E_i)$ is the $(-1)$-curve on $\Tsi_i^-$.
\end{lemma}
\begin{proof}
The first claim is evident --- since the map $\xi_+:\TM \to M^+$ is a contraction
of divisors $E_i$, and the map $\mu\circ\xi:\TM \to X$ contract each of these divisors
to a point, we conclude that $\mu\circ\xi$ factors through $\xi_+$. Moreover,
since $\xi$ and $\xi_+$ are identities over $X_{sm}$, it follows that
the morphisms $M \to X$ and $M^+ \to X$ coincide over $X_{sm}$.
Finally, note that the fiber of $\TM$ over $x_i$ is the union of $E_i$
and the proper preimage of $\Sigma_i^-$. Since $\xi$ is the blowup of $\Sigma_i^+$,
and $\Sigma_i^+$ intersect $\Sigma_i^-$ transversally at $P_i$, the proper
preimage of $\Sigma_i^-$ is the blowup $\Tsi_i^-$ of $\Sigma_i^-$ at $P_i$.
Note also that $E_i \cap \Tsi_i^- = L_i$ is the fiber of $E_i \to \Sigma_i^+$ over $P_i$
and simultaneously the $(-1)$-curve on $\Tsi_i^-$. Finally, since the map $\xi_+$ is
the contraction of $E_i = \Sigma_i^+\times L_i$ onto $L_i$, hence it doesn't change $\Tsi_i^-$,
so the fiber $M^+_{x_i}$ coincides with $\Tsi_i^-$.
\end{proof}

Note that $\Tsi_i^-$ being the blowup of a plane in a point is isomorphic to
a Hirzebruch surface $F_1$. In particular, it has a canonical contraction
$\Tsi_i^- \to \PP^1$ which induces an isomorphism of the exceptional section $L_i \subset \Tsi_i^-$ onto $\PP^1$.
Denote (the pullback to $\Tsi^-_i$ of) the generator of the Picard group of $\Sigma_i$ by $h$
and the class of the exceptional line $L_i \subset \Tsi^-_i$ by $l$.
Then the class of the fiber of the projection $\Tsi_i^- \to \PP^1$ is $h-l$.

\begin{lemma}\label{ommp}
We have $\omega_{M^+|\Tsi^-_i} \cong \CO_{\Tsi^-_i}(-h-l)$.
\end{lemma}
\begin{proof}
Note that $\omega_{\TM} = \xi^*\omega_M(\sum E_i) = \xi_+^*\omega_{M^+}(2\sum E_i)$.
Hence $\xi_+^*\omega_{M^+} = \xi^*\omega_M(-\sum E_i)$. Hence
$$
\xi_+^*\omega_{M^+|\Tsi^-_i} \cong
\xi^*\omega_{M|\Tsi^-_i} \otimes \CO(-E_i)_{|\Tsi^-_i} \cong
\CO_{\Tsi^-_i}(-h) \otimes \CO_{\Tsi^-_i}(-l) \cong
\CO_{\Tsi^-_i}(-h-l).
$$
But $\xi_+$ is an isomorphism on $\Tsi^-_i$, hence the claim.
\end{proof}

It turns out that $M^+$ has a very simple structure ---  it is a $\PP^1$-fibration over
a small resolution of $X$.


\begin{proposition}\label{mp}
The map $M^+ \to X$ factors through a composition
$\xymatrix@1{M^+ \ar[r]^{\mu_+} & X^+ \ar[r]^{\sigma_+} & X}$,
where the map $\mu_+:M^+ \to X^+$ is a $\PP^1$-fibration and $\sigma^+:X^+ \to X$
is a small resolution of singularities. The restriction of the map
$\mu_+$ to the fiber $M^+_{x_i} = \Tsi_i^-$ coincides with the projection
$\Tsi_i^- \to \PP^1$. The curve $C_i = \mu_+(\Tsi_i^-) \cong \PP^1$
is the exceptional locus of $X^+$ over $x_i \in X$.
\end{proposition}
\begin{proof}
We apply to $M^+$ relative Minimal Model Program over $X$, see~\cite{Na}.
Since the relative MMP commutes with the base change, let us first look at
$M^+ \setminus \bigsqcup \Tsi_i^-$ which is the preimage of $X_{sm}$.
The map $M^+ \setminus \bigsqcup \Tsi_i^- \to X_{sm}$ is a $\PP^1$-fibration,
so its relative Picard group is $\ZZ$, and the relative canonical class is ample,
hence the first (and the last) step of the MMP for $M^+ \setminus \bigsqcup \Tsi_i^-$
is the contraction $M^+ \setminus \bigsqcup \Tsi_i^- \to X_{sm}$.

Now consider what happens over an analytic neighborhood of singular points.
Let $x = x_i$ be one of singular points. Consider an analytic neighborhood $U$ of $x$ in $X$
and its preimage $M^+_U \subset M^+$. Then the relative (over $U$) effective cone of $M^+_U$
is generated by curves in the special fiber $M^+_x = \Tsi^-$, that is by the $(-1)$-curve $L$
and by the fiber of the projection $\Tsi^- \to \PP^1$. By Lemma~\ref{ommp} the canonical class
$K_{M^+/X}$ restricts to $\Tsi^-$ as $-h-l$, hence $L$ is $K$-positive, while the fiber is $K$-negative.
Hence the first step in MMP is the contraction of the ray generated by the fiber
of the projection $\Tsi^- \to \PP^1$.
By MMP this contraction should be either
\begin{enumerate}
\item a flip, or
\item a divisorial contraction, or
\item a conic bundle.
\end{enumerate}
By a result of Kawamata~\cite{Ka} the case of a flip is impossible, since the center of a flip
in dimension 4 is always a $\PP^2$, while in our case the only compact surface in $M^+_U$ is $M^+_x = \Tsi^-$
which is a Hirzebruch surface $F_1$.
Similarly, a divisorial contraction is impossible, since then the first step of MMP on $M^+_U \setminus \Tsi_i^-$
would also be a divisorial contraction, while as we have shown above it is a $\PP^1$-fibration.

Thus the first step of MMP for $M^+_U$ is a conic bundle $M^+_U \to X^+_U$.
Once again, over $U \setminus \{x\}$ this conic bundle should coincide with the $\PP^1$-fibration
$M^+ \setminus \bigsqcup \Tsi_i^- \to X_{sm}$, hence gluing all these conic bundles
for all singular points $x_i$, we obtain a global conic bundle
structure on $M^+$, that is a global map $\mu_+:M^+ \to X^+$ for some Moishezon variety $X^+$.
Now we apply~\cite{Te} and conclude that $X^+$ is necessarily smooth, hence $X^+$ is a resolution of singularities of~$X$.
Further, the restriction of the map $\mu_+$ to $\Tsi^-_i$ is a conic bundle which contracts all
the fibers of the projection $\Tsi_i^- \to \PP^1$, hence the fiber of $X_+$ over $x_i$ is the image $C_i$ of $\Tsi_i^-$.
Since $X^+$ is smooth and the fiber of $X^+$ over $x_i$ is $C_i \cong \PP^1$, the map $\sigma_+:X^+ \to X$ is a small
resolution of singularities. So, it remains to check that $\mu_+:M^+ \to X^+$ is a $\PP^1$-fibration.

Since we already know that $\mu_+$ is a conic bundle, we should check that its degeneration locus is empty.
But the degeneration locus of a conic bundle is a divisor, while $M^+ \to X^+$ is nondegenerate
over the complement $X_{sm} = X^+ \setminus (\sqcup C_i)$ of a finite number of curves,
hence the degeneration locus is empty.
\end{proof}

Denoting $f_+ = f\circ\sigma_+$, $\rho_+ = f_+\circ\mu_+$, we obtain a commutative diagram
\begin{equation}\label{dia}
\vcenter{\xymatrix{
& \TM \ar[dl]_\xi \ar[dr]^{\xi_+} \\
M \ar[dd]_\mu  \ar@{-->}[rr]^{\text{flip}} \ar[dr]^\rho && M^+ \ar[dd]^{\mu_+} \ar[dl]_{\rho_+} \\
& Y \\
X \ar[ur]_f && X^+ \ar[ll]^{\sigma_+}  \ar[ul]^{f_+}
}}
\end{equation}

Since the map $\mu_+:\D^b(M^+) \to \D^b(X^+)$ is a $\PP^1$-fibration,
the functor $\mu_+^*:\D^b(X^+) \to \D^b(M^+)$ is fully faithful.
Composing with the functor given by the flip we obtain

\begin{corollary}
The functor $\xi_*\xi_+^*\mu_+^*:\D^b(X^+) \to \D^b(M)$ is fully faithful.
\end{corollary}

Thus we have constructed all the required components in $\D^b(M)$.
It remains to check that they generate the whole category.
This is done in the next section.

\section{Derived category of $M^+$}\label{s-mp}

As it was shown in the previous section, $M^+$ is a $\PP^1$-bundle over $X^+$.
Locally in the \`etale topology this can be represented as a projectivization
of a rank 2 vector bundle. In general, these local bundles do not glue into
a global vector bundle, however their local endomorphism algebras glue and
give a sheaf of Azumaya algebras on $X^+$. We denote this sheaf by $\CB^+$,
it is defined by the $\PP^1$-fibtation up to a Morita equivalence.
The following result of Bernardara~\cite{Be} describes the derived category of $M^+$.

\begin{proposition}[\cite{Be}]\label{dbmp}
There is a semiorthogonal decomposition
$$
\D^b(M^+) = \langle \D^b(X^+), \D^b(X^+,\CB^+) \rangle.
$$
\end{proposition}

So, to prove Theorem~\ref{mt} it remains to check that the functor $\xi_*\xi_+^*$
takes $\D^b(X^+,\CB^+)$ to $\Phi(\D^b(Y,\CB_0))$. For this we construct analogues
of the bundles $\CS_k$ on $M^+$. Recall that $E_i \cong \Sigma_i^+ \times L_i \cong \PP^2\times\PP^1$.

\begin{lemma}
There are vector bundles $\CR_k$ of rank $2$ on $M^+$ such that there is a short exact sequence
\begin{equation}\label{sr}
0 \to \xi^*\CS_k \to \xi_+^*\CR_k \to \oplus \CO_{E_i}(0,-1) \to 0.
\end{equation}
\end{lemma}
\begin{proof}
By Lemma~\ref{sks} we have an isomorphism $(\xi^*\CS_k^\vee)_{|E_i} \cong \CO_{E_i} \oplus \CO_{E_i}(-1,0)$.
Consider the composition $\xi^*\CS_k^\vee \to (\xi^*\CS_k^\vee)_{|E_i} \to \CO_{E_i}(-1,0)$, where the second map
is the unique projection. This map is clearly surjective. Denote the kernel of the sum of these maps over $i$ by $F$,
so that we have an exact triple
$$
0 \to F \to \xi^*\CS_k^\vee \to \oplus \CO_{E_i}(-1,0) \to 0.
$$
Let us check that $F$ is a pullback of a vector bundle from $M^+$.
Since $\xi_+:\TM \to M^+$ is a smooth blowup, it suffices to check that
$F_{|E_i}$ is a pullback of a vector bundle from $L_i$. Let us restrict
the above exact sequence to $E_i$. Since $\CN_{E_i/\TM} \cong \CO_{E_i}(-1,-1)$ we obtain
an exact sequence
$$
0 \to \CO_{E_i}(0,1) \to F_{|E_i} \to \CO_{E_i} \oplus \CO_{E_i}(-1,0) \to \CO_{E_i}(-1,0) \to 0.
$$
The last map is the projection to the second summand, hence we have an exact triple
$$
0 \to \CO_{E_i}(0,1) \to F_{|E_i} \to \CO_{E_i} \to 0.
$$
Since $\Ext^1(\CO_{E_i},\CO_{E_i}(0,1)) \cong H^1(E_i,\CO_{E_i}(0,1)) = 0$,
we see that $F_{|E_i} \cong \CO_{E_i} \oplus \CO_{E_i}(0,1)$.
So, $F_{|E_i}$ is a pullback of $\CO_{L_i} \oplus \CO_{L_i}(1)$,
hence $F$ is a pullback of a vector bundle on $M^+$
which restricts to $L_i$ as $\CO_{L_i} \oplus \CO_{L_i}(1)$.
Now we define $\CR_k$ as the dual of this vector bundle.
So, by definition we have the following exact sequence
\begin{equation}\label{rsd}
0 \to \xi_+^*\CR_k^\vee \to \xi^*\CS_k^\vee \to \oplus \CO_{E_i}(-1,0) \to 0.
\end{equation}
Dualizing this sequence and taking into account that
$$
\RCHom(\CO_{E_i}(-1,0),\CO_\TM) \cong \CO_{E_i}(1,0)\otimes\CN_{E_i/\TM}[-1] \cong \CO_{E_i}(0,-1)[-1]
$$
we obtain~\eqref{sr}.
\end{proof}

The bundles $\CR_k$ enjoy a lot of interesting properties.

\begin{lemma}\label{rhord}
We have $(\rho_+)_*\CR_k^\vee = 0$.
\end{lemma}
\begin{proof}
Indeed, using commutativity of~\eqref{dia} we deduce
$$
(\rho_+)_*\CR_k^\vee =
(\rho_+)_*(\xi_+)_*\xi_+^*\CR_k^\vee =
\rho_*\xi_*\xi_+^*\CR_k^\vee.
$$
Applying the functor $\rho_*\xi_*$ to~\eqref{rsd} we obtain a triangle
$$
\rho_*\xi_*\xi_+^*\CR_k^\vee \to \rho_*\xi_*\xi^*\CS_k^\vee \to \oplus \rho_*\xi_*\CO_{E_i}(-1,0).
$$
The second term equals to $\rho_*\CS_k^\vee$ which is zero by Corollary~\ref{rhosd}.
Since $\rho\circ\xi$ contracts $E_i$ to the point $y_i$, the third term is
$\oplus H^\bullet(E_i,\CO_{E_i}(-1,0))\otimes\CO_{y_i}$,
so it is also zero. Hence $(\rho_+)_*\CR_k^\vee = 0$.
\end{proof}

\begin{proposition}
The bundle $\CR_k$ restricts to any fiber of $\mu_+:M^+ \to X^+$ as $\CO(1) \oplus \CO(1)$.
Moreover
\begin{equation}\label{rsi}
{\CR_k}_{|\Tsi^-_i} \cong \CO_{\Tsi^-_i}(h) \oplus \CO_{\Tsi^-_i}(l).
\end{equation}
\end{proposition}
\begin{proof}
We are going to prove instead that $\CR_k^\vee$ restricts to all fibers as $\CO(-1)\oplus\CO(-1)$.
For fibers of $\mu^+:M^+\setminus \bigsqcup\Tsi^-_i \to X^+ \setminus \bigsqcup C_i$ this follows from Lemma~\ref{rhord}.
So it remains to investigate the restriction of $\CR_k^\vee$ to $\Tsi^-_i$. For this we restrict~\eqref{rsd}:
$$
0 \to (\CR_k^\vee)_{|\Tsi^-_i} \to \CO_{\Tsi^-_i} \oplus \CO_{\Tsi^-_i}(-h) \to \CO_{L_i} \to 0.
$$
It follows that either~\eqref{rsi} holds, or
$(\CR_k^\vee)_{|\Tsi^-_i} \cong \CO_{\Tsi^-_i}(-h-l) \oplus \CO_{\Tsi^-_i}$.
In the former case we are done since both $\CO_{\Tsi^-_i}(h)$ and $\CO_{\Tsi^-_i}(l)$
restrict as $\CO(1)$ to any fiber of $\Tsi^-_i$ over $C_i$.
Let us check that the case $(\CR_k^\vee)_{|\Tsi^-_i} \cong \CO_{\Tsi^-_i}(-h-l) \oplus \CO_{\Tsi^-_i}$ is impossible.

For this we note that by Lemma~\ref{rhord}
$$
0 = \Ext^\bullet(\CO_{y_i},(\rho_+)_*\CR_k^\vee) = \Ext^\bullet(\rho_+^*\CO_{y_i},\CR_k^\vee).
$$
On the other hand, the cohomology sheaves $\CH^l = \CH^l(\rho_+^*\CO_{y_i})$ are supported on $\Tsi^-_i$.
Moreover, $\CH^l = 0$ for $l > 0$ (since $\rho_+^*$ is right exact) and $\CH^0 \cong \CO_{\Tsi^-_i}$.
Consider the spectral sequence
$$
\Ext^q(\CH^p,\CR_k^\vee) \Rightarrow \Ext^{q-p}(\rho_+^*\CO_{y_i},\CR_k^\vee) = 0.
$$
Note that by Serre duality on $M^+$ we have
$$
\Ext^q(\CH^i,\CR_k^\vee) \cong
\Ext^{4-q}(\CR_k^\vee,\CH^i\otimes\omega_{M^+})^\vee \cong
H^{4-q}(M^+,\CH^i\otimes\CR_k\otimes\omega_{M^+})^\vee.
$$
The RHS vanishes for $q \not\in \{2,3,4\}$ since the sheaf
$\CH^i$ is supported on $\Tsi_i^-$ and $\dim\Tsi_i^- = 2$.
Hence the line $q = 3$ doesn't change in the spectral sequence.
But if $(\CR_k^\vee)_{|\Tsi^-_i} \cong \CO_{\Tsi^-_i}(-h-l) \oplus \CO_{\Tsi^-_i}$ then
\begin{multline*}
\Ext^3(\CH^0,\CR_k^\vee) \cong
H^1(M^+,\CO_{\Tsi^-_i}\otimes\CR_k\otimes\omega_{M^+})^\vee \cong \\ \cong
H^1(\Tsi^-_i,(\CO_{\Tsi^-_i}(h+l) \oplus \CO_{\Tsi^-_i})\otimes \CO_{\Tsi^-_i}(-h-l)) \cong \\ \cong
H^1(\Tsi^-_i,\CO_{\Tsi^-_i} \oplus \CO_{\Tsi^-_i}(-h-l)) \cong \kk
\end{multline*}
gives a nontrivial element in $\Ext^3(\rho_+^*\CO_{y_i},\CR_k^\vee)$, which is impossible.
\end{proof}

Now let us describe the Morita equivalence class of $\CB^+$ in terms of $\CR_k$.

\begin{proposition}
There is an Azumaya algebra $\CB'$ on $X^+$ which is Morita-equivalent to $\CB^+$ and such that we have
$\CEnd(\CR_0) \cong \mu_+^*\CB'$, an isomorphism of algebras.
\end{proposition}
\begin{proof}
Note that $\CEnd(\CR_0)$ restricts trivially to all fibers of the $\PP^1$-fibration $M^+ \to X^+$.
Therefore $\CEnd(\CR_0) \cong \mu_+^*\CB'$, where $\CB' = (\mu_+)_*\CEnd(\CR_0)$.
So, we have to check that $\CB'$ is Morita-equivalent to $\CB^+$.
Consider the sheaf $\CF = (\mu_+)_*\CR_0$. Let us check that it has a structure
of a $\CB'-\CB^+$-bimodule and gives a Morita-equivalence.

Choose an \`etale covering $u:U \to X^+$ such that $M^+_U = M^+\times_{X^+}U$ is a projectivization
of a vector bundle $\CE$ on $U$. Let $u:M^+_U \to M^+$ and $\mu_U:M^+_U \to U$ denote the projections:
$$
\xymatrix{
M_U^+ \ar[r]^u \ar[d]_{\mu_U} & M^+ \ar[d]^{\mu_+} \\ U \ar[r]^u & X^+
}
$$
Note that the bundle $u^*\CR_0\otimes\CO_{M^+_U/U}(-1)$ restricts trivially to all fibers
of $M^+_U$ over $U$. Hence there exists a bundle $\CR_U$ on $U$ such that
$$
u^*\CR_0 \cong \mu_U^*\CR_U \otimes \CO_{M^+_U/U}(1).
$$
Therefore $u^*(\mu_+)_*\CR_0 \cong (\mu_U)_*u^*\CR_0 \cong \CR_U\otimes\CE^*$.
On the other hand,
\begin{multline*}
u^*\CB' \cong
u^*(\mu_+)_*\CEnd(\CR_0) \cong
(\mu_U)_*u^*\CEnd(\CR_0) \cong \\ \cong
(\mu_U)_*\CEnd(\mu_U^*\CR_U \otimes \CO_{M^+_U/U}(1)) \cong
(\mu_U)_*\mu_U^*\CEnd(\CR_U) \cong
\CEnd(\CR_U),
\end{multline*}
while $u^*\CB^+ \cong \CEnd(\CE)$ by definition of $\CB^+$.
We see that $u^*(\mu_+)_*\CR_0$ gives a Morita equivalence between
$u^*\CB'$ and $u^*\CB^+$. Since $u:U \to X^+$ is an \`etale covering,
we conclude that $\CB'$ and $\CB^+$ are Morita-equivalent as well.
\end{proof}

From now on we replace $\CB^+$ by a Morita-equivalent Azumaya algebra $\CB'$.
This change does not spoil the decomposition of Theorem~\ref{dbmp},
it still holds after the change. On the other hand, after the change
we have
\begin{equation}\label{erbp}
\CEnd(\CR_0) \cong \mu_+^*\CB^+.
\end{equation}
Now we can relate (this new) $\CB^+$ and $\CB_0$.

\begin{lemma}
We have $(f_+)_*\CB^+ \cong \CB_0$, an isomorphism of sheaves of algebras.
\end{lemma}
\begin{proof}
First of all
$$
(f_+)_*\CB^+ \cong
(f_+)_*(\mu_+)_*(\xi_+)_*\xi_+^*\mu_+^*\CB^+ \cong
f_*\mu_*\xi_*\xi_+^*\CEnd(\CR_0) \cong
\rho_*\xi_*\xi_+^*\CEnd(\CR_0).
$$
On the other hand, tensoring~\eqref{sr} with $\xi^*\CS_0^\vee$ and~\eqref{rsd} with $\xi_+^*\CR_0$ we obtain sequences
$$
0 \to \xi^*\CEnd(\CS_0) \to \xi^*\CS_0^\vee\otimes\xi_+^*\CR_0 \to \oplus (\CO_{E_i}(0,-1) \oplus \CO_{E_i}(-1,-1)) \to 0,
$$
$$
0 \to \xi_+^*\CEnd(\CR_0) \to \xi^*\CS_0^\vee\otimes\xi_+^*\CR_0 \to \oplus (\CO_{E_i}(-1,0) \oplus \CO_{E_i}(-1,-1)) \to 0.
$$
Pushing these sequences along $\rho\circ\xi$ and noting that $E_i$ is contracted to a point we conclude that
$$
\rho_*\xi_*\xi_+^*\CEnd(\CR_0) \cong
\rho_*\xi_*(\xi^*\CS_0^\vee\otimes\xi_+^*\CR_0) \cong
\rho_*\xi_*\xi^*\CEnd(\CS_0).
$$
On the other hand $\rho_*\xi_*\xi^*\CEnd(\CS_0) \cong \rho_*\CEnd(\CS_0) \cong \CB_0$ (see \ref{rhosds} for the last isomorphism).
\end{proof}

The isomorphism $(f_+)_*\CB^+ \cong \CB_0$ gives by adjunction a morphism $f_+^*\CB_0 \to \CB^+$
which equips $\CB^+$ with a structure of a $\CB_0$-module. Now we define the functor
$\D^b(Y,\CB_0) \to \D^b(X^+,\CB^+)$ by $F \mapsto F\otimes_{\CB_0}\CB^+$.
The right adjoint functor then is $(f_+)_*:\D^b(X^+,\CB^+) \to \D^b(Y,\CB_0)$.
Their composition takes $F$ to
$$
(f_+)_*(f_+^*F \otimes_{\CB_0}\CB^+) \cong
F \otimes_{\CB_0}(f_+)_*\CB^+ =
F \otimes_{\CB_0} \CB_0 = F.
$$
This implies that this functor is fully faithful, and the orthogonal
to its image consists of all objects $G$ such that $(f_+)_*G = 0$.
Since $f_+$ is the contraction of $(-1,-1)$ curves $C_i$, any object $G$
in $\D^b(X^+)$ such that $(f_+)_*G = 0$ is a complex with cohomology being direct sums of $\CO_{C_i}(-1)$.
Since $\CB^+$ is an Azumaya algebra, the forgetful functor $\D^b(X^+,\CB^+) \to \D^b(X^+)$
commutes with sheaf cohomology, hence each of these direct sums of $\CO_{C_i}(-1)$ should be a $\CB^+$-module.
So, it remains to check that there is no such $\CB^+$-modules.

For this we note that~\eqref{rsi} implies $\CEnd(\CR_0)_{|\Tsi^-_i} \cong \CO_{\Tsi_i^-} \oplus \CO_{\Tsi_i^-}(h-l) \oplus \CO_{\Tsi_i^-}(l-h) \oplus \CO_{\Tsi_i^-}$.
Therefore we have $\CB^+_{|C_i} \cong \CO_{C_i} \oplus \CO_{C_i}(1) \oplus \CO_{C_i}(-1) \oplus \CO_{C_i} \cong \CEnd(\CO_{C_i} \oplus \CO_{C_i}(-1))$,
hence the category of $\CB^+$-modules supported on $C_i$ is equivalent to the category of sheaves on $C_i$,
the equivalence taking a sheaf $F$  to $F\otimes(\CO_{C_i} \oplus \CO_{C_i}(-1)) \cong F \oplus F(-1)$.
It is clear that if $F \oplus F(-1)$ is a direct sum of $\CO(-1)$ then $F = 0$.
Thus we have checked that $\D^b(Y,\CB_0) \cong \D^b(X^+,\CB^+)$.

Finally, we compute the composition of the equivalence $\D^b(Y,\CB_0) \cong \D^b(X^+,\CB^+)$,
and embeddings $\D^b(X^+,\CB^+) \to \D^b(M^+) \to \D^b(M)$. It acts on $F \in \D^b(Y,\CB_0)$ as
$$
F \mapsto
\xi_*\xi_+^*(\mu_+^*(f_+^*F\otimes_{\CB_0} \CB^+)\otimes_{\CB^+} \CR_0) \cong
\xi_*\xi_+^*(\rho_+^*F\otimes_{\CB_0} \CR_0) \cong
\xi_*(\xi_+^*\rho_+^*F \otimes_{\CB_0} \xi_+^*\CR_0).
$$
Note that $\xi_+^*\rho_+^*F \cong \xi^*\rho^*F$, so tensoring~\eqref{sr} by it we obtain
$$
\xi^*\rho^*F \otimes_{\CB_0} \xi^* \CS_0 \to \xi^*\rho^*F \otimes_{\CB_0} \xi_+^* \CR_0 \to \oplus (\xi^*\rho^*F \otimes_{\CB_0} \CO_{E_i}(0,-1)).
$$
Since $\xi_*\CO_{E_i}(0,-1) = 0$, applying $\xi_*$ we see that
$$
\xi_*(\xi^*\rho^*F \otimes_{\CB_0} \xi^* \CS_0) \cong \rho^*F \otimes_{\CB_0} \CS_0,
$$
which gives the required isomorphism of functors.
This proves Theorem~\ref{mt}.

\section{Concluding remarks and further questions}\label{s-f}

\begin{remark}\label{flop}
There is another way of proving Theorem~\ref{mt}, avoiding use of Moishezon varieties.
For this one has to perform another birational modification of $M$.
First, consider the blowup $M' \to M$ in points $P_i = \Sigma_i^+ \cap \Sigma_i^-$.
Let $E'_i \cong \PP^3$ be the exceptional divisors of this blowup.
Then the proper preimages of the planes $\Sigma_i^+$ and $\Sigma_i^-$ are
Hirzebruch surfaces $\Tsi_i^-,\Tsi_i^+ \subset M'$ which do not intersect.
Moreover, the $(-1)$-curves $L_i^\pm$ on $\Tsi_i^\pm$ are skew-lines in $E'_i$.
One can check that the normal bundle to $\Tsi_i^\pm$ in $M'$ restricts as
$\CO(-1)\oplus\CO(-1)$ to any fiber of $\Tsi_i^\pm$ over $\PP^1$. Hence one can
make a (relative over $\PP^1$) flop in all $2N$ surfaces $\Tsi_i^\pm$ simultaneously.
We will obtain an algebraic variety $M''$ over $X$. The special fibers $M''_{x_i}$
will coincide with blowups $E''_i$ of $E'_i$ in the lines $L_i^\pm$ (and each of the surfaces
$\Tsi_i^\pm$ will be replaced by $\PP^1\times\PP^1$, coinciding with the exceptional
divisor of $E''_i$ over $L_i^\pm$). Then by the same arguments as in Proposition~\ref{mp}
one can show that the map $M'' \to X$ factors as a $\PP^1$-fibration $M'' \to X'$,
where $X'$ is the blowup of $X$ in all points $x_i$. Then a careful analysis of the relation
of the categories $\D^b(M'')$ and $\D^b(M)$ allows to prove Theorem~\ref{mt}.
\end{remark}

An interesting question for the further investigation is to describe the category $\D^b(M)$
without restrictions on the dimension of $Y$. A natural approach would be to consider first
the universal family of quadrics and then to apply a base change argument (see~\cite{K07})
to obtain a decomposition in general case. Unfortunately, the approach of this paper
does not work in this general setup because of the following effect --- assume for simplicity
that $D_3 = \emptyset$, but $\dim D_2 > 0$. Since the fibers of $M$ over $D_2$ are the unions
of two planes, we have an unramified double covering $\tilde{D_2} \to D_2$.
This covering in general is connected. Therefore, we cannot pick up one of the planes
$\Sigma^+$ in all the fibers over $D_2$ and make a flip in them. However, the approach
suggested in~\ref{flop} may work, and I guess that in case $D_3 = \emptyset$ should
work without big changes. As for the case of nonempty $D_3$, a deeper analysis
of the behavior of $M$ over $D_3$ (and possibly more birational transformations)
is required.

Another question which may prove interesting is investigation
of the derived category of the relative scheme of lines (or other isotropic Grassmannians)
of a family of quadrics of dimension bigger than 2.

\end{document}